\documentclass[11pt,a4paper]{amsart}
\usepackage{amssymb, amsmath, amsthm, amscd}
\usepackage{enumerate}
\usepackage{hyperref}\hypersetup{colorlinks = true, linkcolor=blue}
\usepackage{tikz}
\usepackage{tikz-cd}
\usetikzlibrary{calc} 
\usetikzlibrary{math}
\usepackage{mathtools}
\usepackage{systeme}
\usepackage{pgfplots}
\usepackage{comment}

\newcommand{\R}{\mathbb{R}}	
\newcommand{\Z}{\mathbb{Z}}

\newcommand{\C}{\mathbb{C}}

\newcommand{\N}{\mathbb{N}}
\newcommand{\Rep}{\mathcal{R}(\Gamma)}
\newcommand{\Frho}{\mathcal{F}_\rho}

\newcommand{\Hrho}{\mathcal{H}_\rho}

\newcommand{\M}{\mathcal{M}}
\newcommand{\Mrho}{\mathcal{M}_\rho}
\newcommand{\SL}{\operatorname{SL}_2(\C)}
\newcommand{\SLn}{\operatorname{SL}_n(\C)}
\newcommand{\SU}{\operatorname{SU}_2}
\newcommand{\PSL}{\operatorname{PSL}_2(\C)}
\newcommand{\Sl}{\mathfrak{sl}_2(\C)}

\newcommand{\Slrho}{\mathfrak{sl}_2^\rho}

\newcommand{\Slk}[1]{\mathfrak{sl}_2(\C[t]/(t^{#1}))}
\newcommand{\GX}{(\SL\times\SL,\SL)}

\newcommand{\Aut}{\operatorname{Aut}}
\newcommand{\Ad}{\operatorname{Ad}}

\newcommand{\Hom}{\operatorname{Hom}}

\newcommand{\diff}{\mathit{diff}}
\newcommand{\Diff}{\mathrm{Diff}}
\newcommand{\Id}{\operatorname{Id}}
\newcommand{\An}{\mathfrak{An}_\C}
\newcommand{\Sch}{\mathfrak{Sch}}

\newtheorem{proposition}{Proposition}
\newtheorem{lemma}{Lemma}
\newtheorem{theorem}{Theorem}
\theoremstyle{definition}
\newtheorem*{remark}{Remark}

\theoremstyle{definition}

\newtheorem{corollary}{Corollary}

\theoremstyle{definition}
\newtheorem*{definition}{Definition}

\title{On the Teichmüller stack of homogeneous space of $\SL$}
\author{Théo JAMIN}
\address{Théo JAMIN\\ Laboratoire Angevin de REcherhe en MAth\'ematiques\\Universit\'e d'Angers, Universit\'e d'Angers\\ F-49045 Angers Cedex, France}
\email{\href{mailto:tjamin@math.univ-angers.fr}{tjamin@math.univ-angers.fr}}
\date{\today}

\begin{document}
\thispagestyle{empty}
\maketitle

\begin{abstract}
    Let $\Gamma$ be a discrete torsion-free co-compact subgroup of $\SL$. E. Ghys has shown in \cite{Ghys} that the Kuranishi space of $\M=\SL/\Gamma$ is given by the germ of the representation variety $\Hom(\Gamma,\SL)$ at the trivial morphism and gave a description of the complex structures given by representations. In this note, we prove that for all \textit{admissible} representation, i.e. which allow to construct compact complex manifold by this description, the representation variety (pointed at this representation), leads to a complete family even at singular points. Hence, we will consider the \textit{(admissible) character stack} $[\Rep^a/\SL]$, where $\Rep^a$ stands for the open subset formed by admissible representations with $\SL$ acting by conjugation on it and show that this quotient stack is an open substack of the Teichmüller stack of $\M$.
\end{abstract}

\section{Introduction}

Let $\Gamma$ be a discrete co-compact subgroup of  $\SL$ and let $\Rep$ be the associated $\SL$-representation variety $\Hom(\Gamma,\SL)$. Take a representation $\rho \in \Rep$ and consider the following right action 
\begin{align}\label{DefinitiondesMrho}
    \Gamma \times \SL \longmapsto \ & \SL\nonumber \\
    (\gamma,x) \longmapsto \ & \gamma \underset{\rho}{\bullet} x= \rho(\gamma)^{-1}x\gamma 
\end{align}
When this action is free and properly discontinuous we say that $\rho$ is \textit{admissible} and we denote by $\Mrho$ the corresponding quotient manifold and by $\Rep^a$ the set of admissible representations. One can show \cite[Lemme\, 2.1,\, p.115]{Ghys} that $\Rep^a$ and $\Rep$ coincide on a open neigborhood of the trivial morphism $\rho_0:\Gamma \to \Id$. Theorem $A$ of \cite[p.115]{Ghys} states that the Kuranishi space of $\SL/\Gamma$ is the analytic germ of algebraic variety $\Rep$ at the trivial morphism $\rho_0:\Gamma \to \Id$. We will show that this result can be extended in a global version:
\begin{theorem}\label{TeichmullerStack}
    The quotient stack
    \[
        [\Rep^a/\SL]
    \]
    where $\SL$ act by conjugaison is an open substack of the Teichmüller stack of $\SL/\Gamma$.
\end{theorem}
This theorem basically follows from two results, the completeness of the tautological family over the representation and the computation of some group of automorphisms of $\Mrho$ (which give the isotropy group of a point in the Teichmüller stack). More rigorously 
\begin{theorem}\label{completeness}
    For any admissible representation $\rho$, the deformation 
    \[
        \{ \Mrho\ \vert\ \rho\in \Rep^a\} \to \Rep^a
    \]
    pointed at $\rho$ is complete.
\end{theorem}
The plan of this article is to review some notions about the geometry of the $\Mrho$, such as $(G,X)$-structure and the admissibility condition on representations given by a work of Guéritaud, Guichard, Kassel and Wienhard \cite{Gueritaud} and completed by Tholozan \cite{Tholozan}. We will conclude the first part with some computations of automorphisms groups, in particular $\Aut^1(\Mrho)$ which leads to the isotropy group of a point in the character stack and prove 
\begin{proposition}
    For any admissible representation $\rho$, the group $\Aut^1(\Mrho)\coloneqq \Aut(\Mrho) \cap \Diff^0(\Mrho)$ is equal to the quotient of centralizer of $\rho(\Gamma)$ in $\SL$ by $\{\pm \Id \}$. 
\end{proposition}
Then, in a second part, after some cohomological considerations we demonstrate theorems \ref{completeness} and \ref{TeichmullerStack} and we briefly discuss the differences between the character stack and character variety, as a GIT quotient. We will also give some local informations through the computation of the Kodaira-Spencer map and results about equivariant transversal slices which, as germs, gives the Kuranishi space of $\Mrho$. To conclude this paper, we give an example of application. 

\section{Geometry of $\protect{\Mrho}$}

\subsection{$(G,X)$-structure}

In this section, we recall some general ideas of $(G,X)$-structure inspired by Ehresmann and developped by Thurston. 

A $(G, X)$-structure on a manifold $M$ is an atlas of charts with values in the model space X and whose transition functions are restrictions of elements of $G$. A $(G,X)$-manifold is a manifold endowed with this structure. Note that every $G$-invariant geometric structure $g$ on $X$, in the sense of Gromov \cite{GromovGeometricStructure}, defines a structure (locally isomorphic to $g$) on $M$. For example, a holomorphic metric $G$-invariant on $X$ defines a holomorphic metric on $M$.

In the case of $\M=\SL/\Gamma$, we have an obvious $\GX$-structure given by left/right translations on $\SL$ and the Killing form on $\Sl$, which is bi-invariant and non-degenerate, induces a holomorphic metric on $\M$ with constant negative curvature, computed in \cite{Ghys}. 
We call a $(G,X)$-morphism between two $(G,X)$-manifold, a morphism between manifolds which is a local diffeomorphism given in charts of the $(G,X)$-structure by an element of $G$. When dealing with the natural morphism from the universal covering $\widetilde{M}$ of a $(G,X)$-manifold $M$ to $X$, one recover the usual notion of developping and holonomy maps: 
\[
    D:\widetilde{M} \to X,\qquad h:\pi_1(M) \to G
\]  
which satisfies $D(\gamma.x)=h(\gamma).D(x)$ for $\gamma\in \pi_1(M)$ and $x\in \widetilde{M}$. The well-known Ehresmann-Thurston principle \cite{Thurston} states that this holonomy map defines a local homeomorphism from the set of marked $(G,X)$-structures on $M$ to the topological quotient $\Hom(\pi_1(M),G)/G$ (see also \cite{Goldman}). In other word, if $M$ is a $(G,X)$-manifold and $h'$ a representation close to the holonomy $h_0$ of $M$, there exists a $(G,X)$-structure on $M$ with holonomy given by $h'$ and two $(G,X)$-structures are equivalent if their corresponding holonomies are conjugated by a small element in $G$. But the topological quotient $\Hom(\pi_1(M),G)/G$ can be quite bad, even non-reduced \cite{MillsonKapovich} and we want to consider it as the stack for the global point of view, see section \ref{sectionStack}.

When the developping map is a diffeomorphism, we say that the $(G,X)$-structure is complete and we can recover it by taking the quotient of the whole $X$ by $h(\pi_1(M))$. The completeness of such a structure on $\M$ is equivalent to the completeness of the holomorphic metric on $\M$, in the sense that all local geodesics can be extended in global geodesics.

A result of Tholozan \cite[Theorem\, 3,\, p.1923]{Tholozan} state, in the particular case of $\SL$, that the set of complete $\GX$-structure form a union of connected component in the set of deformation of this structure. Hence, we cannot have a continuous deformation of a complete $\GX$-structure with non-complete fibers.

\subsection{Admissibility condition}

We refer to \cite{Tholozan}, \cite{kasselproper} or \cite{Gueritaud} for details on properness condition.

In order to construct the Kuranishi space of $\M$, Ghys show that the action \eqref{DefinitiondesMrho} is free and properly discontinuous for, at least, representations that are close to the trivial one (see \cite[Lemma\, 2.1,\, p.115]{Ghys}). This result was widely improved:\\
\textit{\cite[Theorem\, 1.3,\, p.3]{kasselproper}\,
    Assume that $\Gamma$ is residually finite and not a torsion group. Then $\rho\in \Rep$ is admissible if, and only if, for all $R>0$,
    \begin{align}\label{propernesscondition}
        \mu(\gamma)-\mu(\rho(\gamma))>R
    \end{align}
    for almost all $\gamma\in \Gamma$.}
where $\mu : \SL\to \R_+$ is the projection of a fixed Cartan decomposition of $\SL$ given by $\SU A^+ \SU$\footnote{Note that this is not a diffeomorphism, in opposition to the "classical" Cartan decomposition, due to the non-unicity in this decomposition. Only the projection on $A^+$ is uniquely determined. See \cite[Chapitre 9, Theorem 1.1]{Helgason}} on $A^+\simeq \R_+$. This means that $\rho$ is admissible if its image "drift away at infinity" from $\Gamma$.

In this note, $\Gamma$ is the fundamental group of a hyperbolic 3-manifold thus it is residually finite and without torsion. This theorem state for example that each representation with image contained in a compact subset of $\SL$ is admissible. 

Moreover, we have the following key result for this note:
\begin{proposition}\cite[Corollary\, 1.18]{Gueritaud}
    The set of admissible representations $\Rep^a$ is a (classical) open in $\Rep$.
\end{proposition}

\begin{remark}\label{Admissible variety is not Zariski open}
    Actually, Kassel's results are more precise and in particular one can show that $\Rep^a$ is not, in general, a Zariski open. It only happen in the "rigid case" that is to say when all admissible representations are rigids (i.e. they corresponds to isolated points in $\Rep$), see example \ref{example}.
\end{remark}

\subsection{Automorphisms groups}

Let $\phi$ be an automorphism of $\Mrho$ and $\widetilde{\phi}$ its lifted application to the universal cover. We will denote by $L_g$ (\textit{resp.} $R_g$) the left (\textit{resp.} right) translation by $g$ and by $\iota_g$ the conjugation by $g$.

\begin{lemma}\label{LiftAutomorphisms}
    Let $\phi$ be an automorphism of $\Mrho$. Then there exists $g$ and $\delta$ in $\SL$ such that $\widetilde{\phi}=L_g \circ R_\delta$. 
\end{lemma}

\begin{proof}
    This is using a particular case of theorem $B$ in \cite{Ghys}.\\
    Let $\rho \in \Rep$ and let $\phi$ an automorphism of $\Mrho$. This automorphism $\phi$ lifts to a biholomorphism $\widetilde{\phi}$ of $\SL$ such that there exists $\theta \in \Aut(\Gamma)$ such that the $\Gamma$-equivariance of $\widetilde{\phi}$ is 
    \begin{align}\label{GammaEquivariance}
        \widetilde{\phi}(\gamma \underset{\rho}{\bullet} x)= \theta(\gamma) \underset{\rho}{\bullet}\widetilde{\phi}(x),\ \forall \gamma \in \Gamma.        
    \end{align}
    Because $\SL$ has non-trivial center $\{\pm \Id\}$, we apply Mostow's rigidity to $\PSL$ and lift it to $\SL$. Hence, we know that there exists $\Theta$ a continuous group automorphism of $\SL$ and $\epsilon \in \Hom(\Gamma, \{\pm \Id\})$ such that $\theta =\epsilon. \Theta |_{\Gamma}$. Since $\phi$ is a holomorphic function, $\Theta$ as to be so. But, up to conjugation, the only continuous automorphism of $\SL$ is either the identity or the complex conjugation. Hence, $\Theta$ is an inner automorphism.

    Consider another representation morphism $\eta \in \mathcal{R}(\theta(\Gamma))$ such that 
    \[
        \Theta(\rho(\gamma))=\epsilon(\gamma).\eta (\theta(\gamma)),\ \forall \gamma \in \Gamma
    \]
    It is easy to see that $\Theta$ go down to a biholomorphism between $\Mrho$ and $\M_\eta$. In fact,
    \[
        \Theta(\gamma \underset{\rho}{\bullet} x)=\theta(\gamma)\underset{\eta}{\bullet} \Theta(x),\ \forall \gamma \in \Gamma
    \]
    Now, let $\psi=\widetilde{\phi}\circ \Theta^{-1}$, we get:
    \[
        \psi(\gamma \underset{\eta}{\bullet} x)= \gamma\underset{\rho}{\bullet}\psi(x),\ \forall \gamma \in \Gamma
    \]
    E. Ghys has proved that such biholomorphism has to be a left translation by some element $g$ of $\SL$ such that $\eta$ and $\rho$ are conjugate by $g$. 

    As $\Theta=\iota_\delta$ and $\psi=\textit{L}_h$, for some $\delta$ and $h$ in $\SL$, we have $\widetilde{\phi}(x)=\psi \circ \Theta (x)=h \delta x \delta^{-1}$. Back to the equivariance condition \eqref{GammaEquivariance}, we get successively
    \begin{align*}
        \widetilde{\phi}(\gamma \underset{\rho}{\bullet}x)=&\big{(}\epsilon(\gamma)\iota_\delta(\gamma)\big {)}\underset{\rho}{\bullet} \widetilde{\phi}(x)\\
        h\delta\big{(}\rho(\gamma)^{-1}x\gamma\big{)}\delta^{-1}=&\rho(\epsilon(\gamma)\iota_\delta(\gamma))^{-1}\big{(}h\delta x\delta^{-1}\big{)}\epsilon(\gamma)\delta\gamma\delta^{-1}
    \end{align*}
    Which simplify in
    \begin{align}\label{condition}
        \rho(\epsilon(\gamma)).\rho(\iota_\delta(\gamma))=\epsilon(\gamma).\iota_{g}(\rho(\gamma)),\ \forall \gamma \in \Gamma        
    \end{align}
    where $g=h\delta$.
\end{proof}

Denote by $G_\rho$ the set of pairs $(g,\delta) \in \SL \times \SL$ for which $x\mapsto L_g\circ R_\delta (x)$ descends to an automorphism of $\Mrho$, i.e. pairs $(g,\delta)$ which satisfies \eqref{condition} for some $\epsilon \in \Hom(\Gamma,\{ \pm \Id \})$. As $L_g\circ R_\delta=L_{-g}\circ R_{-\delta}$, we will consider the quotient $PG_\rho \coloneqq G_\rho/\{\pm \Id\}$

\begin{lemma}\label{automorphism}
    Let $\rho\in \Rep^a$ then, we have a surjective morphism of group
    \[
        PG_\rho \to \Aut(\Mrho)
    \]
    with kernel given by Deck transformations, i.e. isomorphic to $\Gamma$.
\end{lemma}
    
\begin{proof}
    The morphism $PG_\rho \ni (g,\delta) \mapsto \phi\in \Aut(\Mrho)$ such that $\widetilde{\phi}=L_g\circ R_\delta$ is surjective by definition of $PG_\rho$ and by previous lemma. As $\SL$ is simply connected, $\phi$ is the identity in $\Aut(\Mrho)$ if, and only if, $\widetilde{\phi}$ is a Deck transformation. That is, $\widetilde{\phi}(x)=\gamma \underset{\rho}{\bullet}x$ for some $\gamma\in \Gamma$, or equivalently $(g,\delta)=(\rho(\gamma)^{-1},\gamma)$.
\end{proof}

\begin{lemma}\label{aut0}
    Let $\rho\in \Rep^a$, then the connected component of the automorphism group of $\Mrho$ is the projection on $\PSL$ of the centralizer $C_{\SL}(\rho(\Gamma))$ of $\rho(\Gamma)$ in $\SL$.
\end{lemma}

\begin{proof}
    By Mostow's theorem, $\Aut(\Gamma)$ is discrete and so is the projection on the second factor of $PG_\rho$. Hence, we get a injection $PG_\rho^0 \to \PSL \times \{\Id\}$. 
    It is straighforward to check that the action of $\Gamma$ on $PG_\rho$ induced by composition of automorphism is given by
    \[
        \Gamma\times PG_\rho \to PG_\rho,\qquad (\gamma,(g,\delta)) \mapsto (g\rho(\gamma),\gamma^{-1}\delta)
    \]
    thus, there is no element of $\Gamma\setminus \{\Id \}$ fixing the connected component of $PG_\rho$. Moreover, the condition \eqref{condition} apply to $(g,\Id)$ is equivalent to require $g$ to be in centralizer of $\rho(\Gamma)$.
    
    Finally, by previous lemma we have 
    \[
        \Aut^0(\Mrho)\simeq \left( PG_\rho / \Gamma\right)^0\simeq PG_\rho^0 \simeq \left(C_{\SL}(\rho(\Gamma))/\{ \pm \Id\}\right)^0
    \]
    And one can check that the centralizer $C_{\SL}(\rho(\Gamma))$ is always connected.
\end{proof}

As in \cite{Meersseman2}, we denote by $\Aut^1(\Mrho)$ the group of automorphisms isotopic to identity through $\mathcal{C}^\infty$-diffeomorphisms (eventually not through biholomorphisms), that is $\Aut^1(\Mrho)=\Aut(\Mrho) \cap \Diff^{\, 0}(\Mrho)$. This group will be used in the next section as it is the isotropy group of a point in the Teichmüller stack. Note that there exists examples of manifolds $X$ for which $\Aut^0(X)\neq \Aut^1(X)$, see \cite{Meersseman2}.

\begin{proposition}\label{aut1}
    Let $\rho \in \Rep^a$, then $\Aut^1(\Mrho) = \Aut^0(\Mrho)$.
\end{proposition}

\begin{proof}
    Let $\phi\in \Aut^1(\Mrho)$ then by lemma \ref{LiftAutomorphisms}, there exists $g$ and $\delta$ in $\SL$ such that $\widetilde{\phi}=L_g\circ R_\delta$. Suppose we have an isotopy 
    \[
        \Phi:\SL \times [0,1] \to \SL
    \]
    with $\Phi(-,t)\in \Diff(\Mrho),\ \Phi(-,1)=\phi$ and $\Phi(-,0)$ is the identity. Obviously, $\Phi_t\coloneqq \Phi(-,t)$ have to preserve fibers (and also its inverse) so that there exists for each $t\in [0,1]$ a corresponding automorphism $\theta_t$ of $\Gamma\simeq \pi_1(\Mrho)$ such that the fibers-preserving condition is
    \[
        \Phi_t(\gamma \underset{\rho}{\bullet}x)=\theta_t(\gamma)\underset{\rho}{\bullet}\Phi_t(x),\ \forall t\in [0,1],\forall \gamma\in \Gamma
    \]  
    By discretness of $\Aut(\Gamma)$, general continuity argument shows that $\theta_t$ is constant and by assumption on $\Phi_0$, it is the identity. Hence, with the same notations as in lemma \ref{LiftAutomorphisms}, the lifted continuous automorphisms $\Theta$ (such that $\theta=\epsilon\Theta$) is also the identity and $\Theta=\iota_\delta=\Id$. We conclude that $\delta=\Id$ and the fiber-preserving condition applied on $g$ gives the same constraint on it that $\Aut^0(\Mrho)$ does.
\end{proof}

\section{representation variety.}\,

As $\Gamma$ arises as a fundamental group of an hyperbolic compact manifold it is finitely presented. Let 
\[
    \langle \gamma_1,\cdots, \gamma_n\vert\, R_1,\cdots,R_m \rangle
\]
be a presentation of $\Gamma$. Thus, we define the representation variety 
\[
    \Rep \coloneqq \left\{ (g_1,\cdots,g_n)\in \SL^n\vert\, R_i(g_1,\cdots,g_n)=\Id,\ \forall 1\leq i \leq m \right\}
\]
This set has a structure of algebraic variety (since $\SL$ is algebraic) and carries two topologies, the Zariski topology and the classical one. Note that, up to isomorphism, the presentation does not change the structure of algebraic variety of $\Rep$.

Let review the Weyl's constuction. Let $\rho_t$ be a smooth path of representation with extremity $\rho$ in $\Rep$. By setting
\[
    c(\gamma) \coloneqq \left.\frac{d\, \rho_t(\gamma) }{dt}\right\vert_{t=0}\rho(\gamma)^{-1}
\]
we obtain a cocycle $c\in Z^1(\Gamma,\Slrho)$, where $\Slrho$ stands for the Lie algebra $\Sl$ with the structure of $\Gamma$-module given by the adjoint representation composed by $\rho$. If $\rho_t$ is given by the conjugaison of $\rho$ by a path of matrices $A_t$ emanating from $\Id$ the corresponding cocycle is a coboundary, i.e. given by $\gamma\mapsto X-\Ad_{\rho(\gamma)}X$ with $X=\left.\frac{d\, A_t }{dt}\right\vert_{t=0}$. This construction leads to:\\
\textit{We have the following isomorphism \cite[Proposition\, 2.2]{LubotzkyMagid} 
\[
    T_\rho^{Zar} \Rep \simeq Z^1(\Gamma,\Slrho)
\]
and the inclusion $\mathcal{I}\subset\sqrt{\mathcal{I}}$ induces an injection 
\[
    T_\rho^{Zar} \Rep_{red} \hookrightarrow Z^1(\Gamma,\Slrho)
\]} 
where $\Rep_{red}$ is the reduction of the affine scheme $\Rep$ and $\mathcal{I}$ is the ideal defining the variety $\Rep$. This inclusion can be strict, see \cite[Example\, 2.18]{Heusener}. As we are interested in compute Kuranishi spaces, which can be non-reduced, it is very important to deal with the scheme $\Rep$ and not its reduction.

\begin{remark}
    Actually, Kapovich and Millson \cite{MillsonKapovich} proved that there are no “local” restrictions on geometry of the $\SL$-representation schemes of 3-manifold groups. 
\end{remark}

\section{Teichmüller stack}

\subsection{Preliminary results}

Let $\rho\in \Rep$ be admissible. The tangent bundle of $\Mrho$ is identified with the adjoint bundle associated to the $\SL$-principal bundle $\pi_\rho:\SL \to \Mrho$, the universal cover, that is $T\Mrho \simeq \Ad_\rho(\SL) \coloneqq \SL \times_{Ad_\rho}\Sl$ where the action is given by
\begin{align}
    \Gamma \times \SL\times \Sl \longrightarrow \ & \SL\times \Sl \nonumber\\ 
    \left(\gamma,(x,v)\right) \longmapsto \ & \left(\rho(\gamma)^{-1}x\gamma, \Ad_{\rho(\gamma)^{-1}}(v)\right) \label{action}
\end{align}
Consider the sheaf $\Theta_\rho$ given by germs of its holomorphic sections. Remark that holomorphic sections of this bundle corresponds to holomorphic vector fields on $\Mrho$, it follows that $\Theta_\rho$ is exactly the sheaf of germs of holomorphic vector fields on $\Mrho$. 

It is well known that this tangent bundle, as it is constructed by a representation of the fundamental group, carries a flat connection (see for example \cite{GoldmanMillsonKahler}). We also denote by $\Frho$ the sheaf of germs of its flat sections. The interest of these sheaves is that $H^1(\Mrho,\Theta_\rho)$ is identified to the Zariski tangent of the Kuranishi space at the base point and the elements of $H^1(\Mrho,\Frho)$ corresponds to infinitesimal deformations of the $\GX$-structure of $\Mrho$.

\begin{proposition}\label{embedding}
    Let $\rho$ be an admissible representation. Then, the embedding of $\Frho$ in $\Theta_\rho$ induces an isomorphism 
    \[ 
        H^1(\Mrho,\Frho) \simeq H^1(\Mrho,\Theta_\rho)\label{isoH1}
    \]
    and an injection
    \[
        H^2(\Mrho,\Frho) \hookrightarrow H^2(\Mrho,\Theta_\rho)\label{plongementH2}
    \]
\end{proposition}

We will show the successive maps:
\[
    H^i(\Mrho,\Frho)\simeq H^i(\Gamma,\Slrho) \hookrightarrow H^i(\Gamma,\Hrho)\simeq H^i(\Mrho,\Theta_\rho),\ i\geq 0
\]
where $\Hrho$ is the set of global holomorphic functions with values in $\Slrho$. Here, $\Slrho$ stands for $\Sl$ endowed with the structure of $\Gamma$-module induced by \eqref{DefinitiondesMrho}, i.e. given by $\Ad\circ \rho$. Then we will prove that the embedding is actually an isomorphism for $i=0$ and $1$.

\begin{lemma}\label{Cechtogroupcohomology}
    Let $\rho\in \Rep^a$, then
    \[
        H^i(\Mrho,\Frho)\simeq H^i(\Gamma,\Slrho),\hookrightarrow H^i(\Gamma,\Hrho) \simeq H^i(\Mrho,\Theta_\rho),\ i\geq 0
    \]
\end{lemma}

\begin{proof}
    The way to go from \v Cech coholomogy to group cohomology is given by a well known result in \cite[Appendix to §2,\, p.22]{Mumford}.
    Consider the case $\pi_\rho:\SL\to \Mrho$ and $\mathcal{F}$ is $\Frho$ or $\Theta_\rho$. As both sheaves are obtained as sheaves of germs of sections of fiber bundles, the pullback sheaves are simply the corresponding sheaves of germs of sections of the pullback bundles:

    \begin{center}
        \begin{tikzcd}
            \SL\times\Sl \ar[r, "d\pi"] \ar[d, "p_1"] & \SL\times_{\Ad_\rho}\Sl \ar[d, "\pi_\rho"]\\
            \SL \ar[r, "\pi"] & \Mrho
        \end{tikzcd}
    \end{center}

    Therefore, the global holomorphic sections (\textit{resp.} flat sections) of the trivial bundle $\SL\times \Sl \to \SL$ is the set of holomorphic (\textit{resp.} constant) functions from $\SL$ to $\Sl$, which we denoted by $\Hrho$ (\textit{resp.} $\Slrho$).
    The $\Gamma$-structure of both sets is given by precomposition by the action of $\Gamma$ via $\underset{\rho}{\bullet}$ and postcomposition by adjoint representation of $\rho$, that is 
    \begin{align}\label{GammaStructureHrho}
        \Hrho \ni f \mapsto \left(\gamma.f:x\mapsto \Ad_{\rho(\gamma)^{-1}}f(\rho(\gamma^{-1})x\gamma) \right)    
    \end{align}
    The Cartan's theorem $B$ states that for any Stein manifold $X$ and any coherent sheaf $\mathcal{F}$, $H^p(X,\mathcal{F})$ vanish for $p\geq 1$. In our context, $\SL$ is a Stein manifold as it is isomorphic to the affine variety $ad-bc=1$ in $\C^4$ and the sheaves $\Theta_\rho$ and $\Frho$ are locally free. We finally end up with the isomorphisms
    \[
        H^i(\Gamma,\Slrho) \simeq H^i(\Mrho,\Frho), \quad H^i(\Gamma,\Hrho) \simeq H^i(\Mrho,\Theta_\rho),\ \forall i\in \N
    \]

    Finally, as the embedding of $\Slrho$ in $\Hrho$ is $\SL$-equivariant, by general arguments in group cohomology \cite{Brown}, the applications 
    \[
        H^i(\Gamma,\Slrho) \to H^i(\Gamma,\Hrho)
    \]
    are injective.
\end{proof}

\begin{proof}[Proof of proposition \ref{embedding}]
    Let as always $\rho$ be an admissible representation. Consider the short exact sequence of $\Gamma$-modules
    \[
        0\to \Slrho \to \Hrho \to \Xi_\rho \coloneqq \Hrho/(1\otimes\Slrho) \to 0
    \]
    and the following part of the long exact sequence associated to it
    \begin{center}
        \begin{tikzcd}
            H^0(\Gamma,\Xi_\rho) \ar[r, "\delta_0"] & H^1(\Gamma,\Slrho) \ar[r, "f^1"]& H^1(\Gamma,\Hrho) \hspace{0.5cm} \ar[dll, out=-30, in=150] \\
            \hspace{1cm} H^1(\Gamma,\Xi_\rho) \ar[r, "\delta_1"] & H^2(\Gamma,\Slrho) \ar[r, "f^2"] & H^2(\Gamma,\Hrho)  
        \end{tikzcd}    
    \end{center}
    By lemma \ref{Cechtogroupcohomology}, $f^1$ and $f^2$ are injective maps, so the coboundary maps $\delta_i$ are the zero maps. We end up with the short exact sequence
    \begin{align}\label{ExactsequenceConstantfunctions}
        0 \to H^1(\Gamma,\Slrho) \to H^1(\Gamma,\Hrho) \to H^1(\Gamma,\Xi_\rho) \to 0        
    \end{align}
    We will show that the last term always vanish.
    
    Consider the two $\Gamma$-modules $\Xi_\rho$ and $\Xi_{\rho_0}$ underlying the same abelian group (identify with global holomorphic functions from $\SL$ to $\Sl$ with vanishing constant term) but with $\Gamma$-module structures induced respectively by $\rho$ and $\rho_0$, the trivial morphism. From the short exact sequence of groups:
    \begin{align}
        1 \mapsto \Gamma_0 \coloneqq \ker(\rho) \longmapsto \Gamma \longmapsto \rho(\Gamma) \longmapsto 1 \label{exactsequence}
    \end{align}
    one can construct the two associated inflation-restriction exact sequences respectively with values in $\Xi_\rho$ and $\Xi_{\rho_0}$:
    \begin{center}
        \begin{tikzcd}
            0 \ar[]{r} &  H^1(\rho(\Gamma), \Xi_\bullet^{\Gamma_0}) \ar[]{r} & H^1(\Gamma, \Xi_\bullet)\hspace{0.5cm} \ar[out=-30, in=150, "res"']{dl} & \\
            & \hspace{1cm} H^1(\Gamma_0, \Xi_\bullet)^{\rho(\Gamma)} \ar[]{r}& H^2(\rho(\Gamma), \Xi_\bullet^{\Gamma_0}) \ar[]{r} & H^2(\Gamma,\Xi_\bullet)
        \end{tikzcd}    
    \end{center}
    for $\bullet=\rho$ or $\rho_0$.
    Obviously, the action of $\Gamma_0=\ker(\rho)$ on $\Xi_\bullet$ is the same for $\bullet=\rho$ or $\rho_0$, which is given by precomposition by right multiplication by $\gamma\in\Gamma_0$ (see \eqref{GammaStructureHrho}). Moreover, E. Ghys showed \cite[p.131-132]{Ghys} that a holomorphic function invariant by $\Gamma_0$ is also invariant by its Zariski closure, which is $\SL$ by \cite[Lemma\, 5.6]{Ghys}. Hence, these functions are constant and by definition of $\Xi_\bullet$, equal to zero. We end up with $\Xi_\rho^{\Gamma_0}=\Xi_{\rho_0}^{\Gamma_0}=0$ and it follows that $H^1(\rho(\Gamma),\Xi_\bullet^{\Gamma_0})=H^2(\rho(\Gamma),\Xi_\bullet^{\Gamma_0})=0$. In other words, the restriction map, which is the curvy arrow in the previous inflation-restriction exact sequence, is an isomorphism either for $\rho$ and $\rho_0$. 
    to summarize, we have the following isomorphisms
    \begin{align*}
        H^1(\Gamma_0,\Xi_\rho) \simeq & H^1(\Gamma_0,\Xi_{\rho_0})\\
        res:H^1(\Gamma,\Xi_\rho) \overset{\sim}{\longrightarrow} & H^1(\Gamma_0,\Xi_\rho)^{\rho(\Gamma)}\\
        res:H^1(\Gamma,\Xi_{\rho_0}) \overset{\sim}{\longrightarrow} & H^1(\Gamma_0,\Xi_{\rho_0})^{\rho(\Gamma)}=H^1(\Gamma_0,\Xi_{\rho_0})
    \end{align*}
    Theorem $4.1$ in \cite{Ghys} states, with our notations, that $H^1(\Gamma,\Xi_{\rho_0})=0$. With the previous isomorphisms and this result we have that $H^1(\Gamma,\Xi_\rho)=0$ as announced and \eqref{ExactsequenceConstantfunctions} gives the desired isomorphism. 
\end{proof}

\subsection{Higher obstructions}

We want to describe deformations of $\Mrho$ over $(\C,0)$. In order to do it, we recall the construction given by \cite[p.4-10]{Douady}. \\
For all open $U$ in $\M_\rho$, we consider biholomorphisms
\[
    f:W \to W'
\]
where $W,W'\subset \M_\rho \times \C$ are open which contains $U\times \{0\}$. We consider the set of such biholomorphisms which preserves the fibers $M\times \{p\}$ and such that $f\vert_{\M_\rho \times \{0\}} =\Id$. We define the sheaf $\Lambda_\rho$ by $\Lambda_\rho(U)$ as the quotient of this set by identify two biholomorphisms which coincides on a neighborhood of $U\times \{0\}$. The important fact is that\\
\textit{The space $H^1(\M_\rho,\Lambda_\rho)$ is identify to the set of classes of germs of deformations of $\Mrho$ parametrized by $(\C,0)$} \cite{Douady}.
    
This sheaf is naturally filtered:\\
For each open $U$, we consider the set of biholomorphisms of $\Lambda_\rho^n(U)$ which are tangent to the identity up to the order $n-1$ and we denote by $\Lambda_\rho^n$ the corresponding sheaf. For all $n\geq 1$, we denote by $Q_\rho^n$ the quotient sheaf $\Lambda_\rho/\Lambda_\rho^{n+1}$. It is well know that (see \cite{Morrow-Kodaira})
\[
    \ker\left(Q_\rho^{n+1} \to Q_\rho^n \right) \simeq \Theta_\rho
\]
Thus, we get the following exact sequence of sheaves
\begin{align}\label{nthorderexactsequence}
    0 \to \Theta_\rho \to Q_\rho^{n+1} \to Q_\rho^n \to 0
\end{align}
The elements of $H^1(\Mrho,Q_{n})$ are called $n$-th order deformation of $\Mrho$.

\begin{proof}[Proof of Theorem \ref{completeness}]
    Assume that up to order $n$, the set of classes of germs of deformation of $\Mrho$ over $\C$ is given by germs of deformations of the representation $\rho$ by cochains $\{c_i\}_{i=1}^n$ via
    \[
        \rho_n\coloneqq\rho^{(c_1,\cdots,c_n)}:\gamma \mapsto \exp\left(\sum_{i=1}^n c_i(\gamma)t^i\right)\rho(\gamma)
    \]
    
    Then, we can equip $\mathfrak{g}_{n}\coloneqq \Slk{n+1}$ with the $\Gamma$-structure given by $\Ad_{\rho_n}$. We denote $\mathfrak{g}_n^{\rho_n}$ the Lie algebra with its $\Gamma$-structure. 
    
    Interpreting $\mathcal{B}_\rho^n\coloneqq H^0(\SL,\pi^*Q_\rho^n)$ as a set of global sections of $n$-jets, we get an injection of $\Gamma$-modules $\mathfrak{g}_n^{\rho_n} \to \mathcal{B}_\rho^n$. These maps induce a morphism between exact sequences
    \begin{center}
        \begin{tikzcd}[column sep=1.4em]
            0 \ar[r]  & \Slrho \ar[r] \ar[d] & \mathfrak{g}_n^{\rho_n} \ar[r] \ar[d] & \mathfrak{g}_{n-1}^{\rho_{n-1}} \ar[r] \ar[d] & 0\\
            0 \ar[r] & \Hrho \ar[r] & \mathcal{B}_\rho^n \ar[r] & \mathcal{B}_\rho^{n-1} \ar[r] & 0
        \end{tikzcd}    
    \end{center}
    which induces, a map of long exact sequences in cohomology of groups 
    \begin{center}
        \begin{tikzcd}
            H^1(\Gamma,\Slrho) \ar[r] \ar[d, "i_1"] & H^1(\Gamma,\mathfrak{g}_n^{\rho_n}) \ar[d, "i_2"] \ar[r] 
            & H^1(\Gamma,\mathfrak{g}_{n-1}^{\rho_{n-1}} ) \ar[d, "i_3"] \ar[r, "\delta"] & H^2(\Gamma,\Slrho) \ar[d, "i_4"] \\
            H^1(\Gamma,\Hrho) \ar[r] \ar[d, "\wr"] & H^1(\Gamma,\mathcal{B}_\rho^n) \ar[r] \ar[d,"\wr"]
            & H^1(\Gamma,\mathcal{B}_\rho^{n-1}) \ar[r] \ar[d,"\wr"] & H^2(\Gamma,\Hrho) \ar[d,"\wr"] \\
            H^1(\Mrho, \Theta_\rho)\ar[r] &H^1(\Mrho,Q_\rho^{n}) \ar[r] 
            & H^1(\Mrho,Q_\rho^{n-1}) \ar[r, "\check{\delta}"] & H^2(\Mrho,\Theta_\rho) \\
        \end{tikzcd}    
    \end{center}
    Proposition \ref{embedding} says that $i_1$ is an isomorphism and $i_4$ is a monomorphism. By assumption, $i_3$ is an isomorphism. The four-lemma states that $i_2$ is surjective, thus it is an isomorphism.
    
    Let $U_n(\gamma)\coloneqq \left( \frac{d}{dt} \rho_n(\gamma) \right)\rho_n(\gamma)^{-1}$ and $f_n$ the corresponding element in $H^1(\Mrho,Q_\rho^{n-1})$. If $f_n$ can be extended to order $(n+1)$ then the class of $\check{\delta} (f_n)$ in $H^2(\Mrho,\Theta_\rho)$ is zero. The class of $[\delta U_n]\in H^2(\gamma,\Slrho)$ is then also zero which is an equivalent condition to the existence of a cochain $c_{n+1}$ such that $\rho_{n+1}\coloneqq \rho^{(c_1,\cdots, c_{n+1})}$ is a morphism up to order $n+1$ (see \cite[Proposition\, 3.1]{HeusenerPorti}). 
    
    Inductively, a deformation of the complex structure of $\Mrho$ parametrized by $(\C,0)$ is given by a formal deformation of the representation $\rho$:
    \[
        \rho_\infty: \gamma \mapsto \exp \left( \sum_{i=1}^\infty c_i(\gamma)t^i \right)\rho(\gamma)
    \]
    The existence of a convergent solution follows directly from a result of Artin \cite{Artin}, as in \cite[Proposition\, 3.6]{HeusenerPorti}.
    This show us that the representation variety is complete at each point that corresponds to an admissible representation and therefore this conclude the proof of theorem \ref{completeness}.
\end{proof}

\subsection{Teichmüller stack}\label{sectionStack}

The Newlander-Nirenberg Theorem \cite{Newlander-Niremberg} says that a structure of a complex
manifold on $M$ is equivalent to a a $C^\infty$ bundle operator $J$ on the tangent bundle of $M$ such that
\[
    J^2=-\Id, \text{ and } [T^{0,1},T^{0,1}]\subset T^{0,1}
\]
Where $T^{0,1} = \{v + iJv \vert\,  v \in TM \otimes \C\}$ is the subbundle given by the eigenvectors of $J$ with eigenvalue $-i$ of the complexified tangent bundle of $M$. 
We denote by $\mathcal{I}(M)$ the set of complex structure on the $C^\infty$ manifold $M^\diff$ (forgetting its natural complex structure). Note that the group $\Diff(\Mrho)$ of $C^\infty$-diffeomorphisms of $M$ act on $\mathcal{I}(M)$ as
\[
    \Diff(M) \times \mathcal{I}(M) \to \mathcal{I}(M), \quad (f,J) \to (df)^{-1}\circ J\circ df
\]
The Teichmüller space of $M$, denoted $\mathcal{T}(M)$, is given by the quotient of $\mathcal{I}(M)$ by the action of the subgroup $\Diff^0(M)$ of $\Diff(M)$ formed by diffeomorphisms isotopic to the identity. There exists example of manifold $M$ for which (see \cite[Example 12.3]{MeerssemanTeichStack}), this topological space does not admit a structure of analytic space. But, under some assumption on the dimension of the group of automorphisms of $M$, the Teichmüller has a structure of Artin stack and we shall review some definitions of its construction.  

\subsection{Teichmüller and character stack}

Let $\An$ be the category of complex analytic space. In this note, a \textit{stack} is a stack in groupoids over the site $\An$ in the sense of \cite[Definition\, 8.5.1]{stacks-project}.
Let $M$ be a $C^\infty$ manifold which admits a complex structure $J$. We construct the Teichmüller stack $\mathcal{T}(M)$ of $M$ as the category whose 
\begin{itemize}
    \item \textit{objects} are deformations of $M$ which are $\Diff^0(M)$-bundle when considered in the $C^\infty$ categorie.\\
    That is, smooth and proper morphism $\pi:X\to B$, between objects $X,B\in (\An)$, which is diffeomorphic, when considered as real analytic spaces, to a bundle $E\to B$ with fiber $M$ and structural group reduced to $\Diff^0(M)$.
    \item \textit{morphisms} are cartesian diagrams 
    \begin{center}
        \begin{tikzcd}
            X' \ar[r] \ar[d,"\pi'"] & X \ar[d, "\pi"]\\
            B' \ar[r, "f"] & B
        \end{tikzcd}    
    \end{center}
    where the isomorphism $f^*X\simeq X'$ induced a $\Diff^0(M)$-isomorphism of the smooth bundle structure in the category of real analytic spaces.
\end{itemize}
If $V$ is an open subset of $\mathcal{I}(M)$, we can define in the same way $\mathcal{T}_V(M)$ the Teichmüller stack of $M$ for complex structures belonging to $V$, that is objects are smooth morphisms $\pi:X\to B$ as well but the complex structures on fibers of $\pi$ belongs to $V$.
For more details see \cite{MeerssemanTeichStack}. 

We want to define a map $i:\Rep^a \to \mathcal{I}(\M)$ which sends an admissible representation $\rho$ to the bundle operator corresponding to the natural complex structure of $\Mrho$. So we can define $\mathcal{T}_{\Rep^a}(\M)$ the Teichmüller stack of $\M$ for complex structures arising as $\Mrho$ for some $\rho\in \Rep^a$.
The way to construct $i$ is the following. Take $\rho \in \Rep^a$ and consider the frame bundle $F(\Mrho^{\diff})$ of $\Mrho^{\diff}$ the $C^\infty$ manifold underlying $\Mrho$. Points in this bundle over $x\in \Mrho$ are identified with linear isomorphisms $\R^6\to T_x\Mrho^{\diff}$. Note that the tangent bundle $\Ad_\rho(\SL)$ (see \eqref{action}) gives a natural subbundle of $F(\Mrho^{\diff})$ by $\C$-linear isomorphisms $\C^3\to T_x\Mrho^{\diff}$ and the corresponding reduction of the structural group is exactly the $C^\infty$ bundle operator $J_\rho$ corresponding to the complex structure of $\Mrho$.
Hence, we define $i$ by $i:\rho \mapsto J_\rho$.

Naturally, we define

\begin{definition}
    The \textit{character stack} (\textit{resp. admissible character stack}) is the quotient stack 
    \[
        [\Rep/\SL], \quad (\textit{resp.}\ [\Rep^a/\SL])
    \]
    over the site $\An$.
\end{definition}

Obviously the admissible character stack is a substack of the character stack in the sense of \cite[Definition\, 6.9,\, p.112]{AlgebraicStacks}, that is a full saturated subcategory of the character stack which is also a stack. 

\begin{remark}\label{Remark analytic vs algebraic}
    It is important to notice that the character stack see as a stack over the site $\Sch$ of schemes is algebraic but the admissible character stack is not since $\Rep^a$ is not a Zariski open in $\Rep$ (see remark \ref{Admissible variety is not Zariski open}). However, both of them are analytic stacks and this explains why we have to work on the analytic site rather than an algebraic one.
\end{remark}

\begin{theorem}
    The admissible character stack is an open substack of the Teichmüller stack of $\M$.
\end{theorem}

\begin{proof}
    The completeness theorem \ref{completeness} implies that there exists an open $V^a \subset \mathcal{I}(\M)$ of complex structures $\Mrho$ given by representations $\rho\in \Rep^a$. Hence, locally we know that any deformation $\mathcal{X}\to B$ in the Teichmüller stack $\mathcal{T}_{V^a}(\M)$ can be seen as a $\SL$-principal bundle $P \to B$ with an $\SL$-equivariant map $p:P\to \Rep^a$, that is an element of the (admissible) character stack.
    
    Denote by $\mathfrak{X}\to \Rep^a$ the tautological family above $\Rep^a$, that is $\mathfrak{X}$ is obtained as the quotient of $\SL\times \Rep^a$ by the action of $\Gamma$ :
    \begin{align*}
        \Gamma\times \SL\times \Rep^a \to& \SL\times \Rep^a \\
        (\gamma,x,\rho)\mapsto& (\rho(\gamma)^{-1}x\gamma,\rho)
    \end{align*}
    We restrict our attention on isomorphism between $\SL$-torsors so we only look at tautological families. Let $B$ be an analytic space and $\phi,\psi:B\to \Rep^a$ analytic maps such that the induced tautological families $\phi^*\mathfrak{X}\to B$ and $\psi^*\mathfrak{X}\to B$ are isomorphic in the Teichmüller stack. So there exists an analytic map $F:\phi^*\mathfrak{X} \to \psi^*\mathfrak{X}$ such that
    \begin{center}
        \begin{tikzcd}
            \phi^*\mathfrak{X} \ar[r,"F"] \ar[d, "\pi_\phi"] & \psi^*\mathfrak{X} \ar[d, "\pi_\psi"]\\
            B \ar[r, "\Id"] & B
        \end{tikzcd}
    \end{center}
    is a cartesian diagram and $F$ is a $\Diff^0$-bundle isomorphism. Lifting $F$ to an analytic map $\widetilde{F}:\SL\times \Rep^a \to \SL\times \Rep^a$, we see that on each fibers $\left. \widetilde{F} \right\vert_{\pi_\phi^{-1}(b)}(x,\rho)=(\iota_g(x),\iota_g\circ \rho)$, where $g\in \Aut^1(\Mrho)\simeq C_{\SL}(\rho(\Gamma))$ by proposition \ref{aut1}. Doing this on each fibers, we obtain a map 
    \[
        f:\Rep^a \to \SL
    \]
    such that $\widetilde{F}(x,\rho)=(\iota_{f(\rho)}(x),\iota_{f(\rho)}\circ \rho)$. This application obviously satisfies $s(f(\rho),\rho)=\rho$ and $t(x,f)=\iota_{f(\rho)}(\rho)$, where the map $s$ and $t$ are the source and the target map of the Lie groupoid 
        \[
            \SL\times \Rep^a \overset{\iota}{\underset{p_2}{\rightrightarrows}} \Rep^a
        \]
    that is $s$ is the projection on the second factor and $t$ is the $\SL$-action of conjugation on $\Rep^a$.
\end{proof}

We easily deduce the following corollary, which is a reformulation of the theorem \ref{TeichmullerStack}:

\begin{corollary}
    The Lie groupoid 
    \[
        \SL\times \Rep^a \overset{\iota}{\underset{p_2}{\rightrightarrows}} \Rep^a
    \]
    is an atlas for $\mathcal{T}_{\Rep^a}(\M)$.
\end{corollary}

\begin{remark}
    However, it is an open question to know if this open substack is a union of connected components of the Teichmüller stack, or if it is not, what is the boundary of this substack.
\end{remark}

\section{Characters stack versus character variety}

We want to emphasize the use of stack instead of GIT quotient. As remarked before (see remark \ref{Remark analytic vs algebraic}), $\Rep^a$ is not a Zariski open when $b_1(\Gamma)\neq 0$, hence it is not possible to form the quotient (in the sense of geometric invariant theory) nor the algebraic stack associated. Actually, the situation seems to be even worse, for instance, let $\pi:\Rep \to \Rep//\SL$ be the affine quotient and take $\rho\in \Rep^a$ such that its orbit is not closed, there is no reason for $\pi^{-1}(\pi(\rho))$ to contains only admissible representations, even if we don't have any example of such situation.

We want to underline the fact that the character stack contains more informations as the character variety. To do so, we can look at fibers of the morphism 
\[
    \phi:[\Rep^a/\SL]\to [\Rep//\SL]
\]
where $X(\Gamma)\coloneqq[\Rep//\SL]$ stands for the stack associated to the affine quotient $\Rep//\SL$, over the site $\An$. Let $\chi$ be a point in $X(\Gamma)$, then the it is easy to see that the preimage of $\chi$ by $\phi$ is formed by all $\SL$-principal bundles $\mathcal{O}(\rho) \to \{\rho\}$ such that $\pi(\rho)=\chi$. Whenever $\chi$ is obtained as the trace character of two non conjugated representations $\rho$ and $\eta$, the preimage of $\phi$ contains two non biholomorphic families in the Teichmüller stack. In other words, there are points in the Teichmüller stack which are identified in the character variety.

\section{Kodaira-Spencer map}

In this section, we will show that the Kodaira-Spencer map associated to the natural deformation over $\Rep^a$ is surjective at each point.   

Consider the variety $\widetilde{\mathfrak{X}}\coloneqq\SL \times \Rep^a$ and its quotient $\mathfrak{X}$ by $\Gamma$ given by the action
\begin{align*}
    \Gamma \times \widetilde{\mathfrak{X}}\ \to \ & \widetilde{\mathfrak{X}} \\
    (\gamma,(x,\rho))\ \mapsto \ & (\rho(\gamma^{-1})x\gamma,\rho)
\end{align*}
The natural projection $p_2:\mathfrak{X} \to \Rep^a$ is a deformation of complex structures with $\Mrho\subset \mathfrak{X}$ above $\rho\in \Rep^a$. Let $\rho$ be an admissible representation and $V$ a Stein open neigborhood in $\Rep^a$ containing $\rho$. One can consider the fundamental exact sequence of this deformation restricted to $V$ 
\begin{align*}
    0\to \Theta|_{p_2^{-1}(V)} \to \Pi|_{p_2^{-1}(V)} \to \Upsilon|_{p_2^{-1}(V)} \to 0
\end{align*}
Where $\Theta$ is the sheaf of germs of holomorphic vector fields on fibers of $p_2$, $\Pi$ is the sheaf of projectable vector fields and $\Upsilon$ the sheaf of germs of vector fields on $\Rep$.\\
And for the infinitesimal neigborhood of $\rho$, this sequence tends to
\begin{align}\label{fundasequence}
    0\to \Theta|_{\Mrho} \to \Pi|_{\Mrho} \to \Upsilon|_{\Mrho} \to 0
\end{align}
It is well known that the Kodaira-Spencer map $\mathcal{KS}_\rho$ of this deformation is the connecting homorphism of the long exact sequence associated to it. 

Let $V$ be an small Stein neigborhood of $\rho$ in $\Rep$. Since $\widetilde{\mathfrak{X}}\vert_V=\SL\times V $ is a product of Stein manifolds it is also Stein. The pullback sheaf is not necesseraly coherent, but $\widetilde{\mathfrak{X}}\vert_V$ is an open in affine and thus noetherian so the pullback of any coherent sheaf if also coherent. Then by Cartan's theorem $H^i(\widetilde{\mathfrak{X}}\vert_V, \pi^*\Theta|_V)=\{0\}$ for $i>0$, in particular for $i=1$. Hence, on the infinitesimal neighborhood of $\rho$, the sequence 
\[
    0 \to H^0(\SL,\pi^*\Theta|_{\Mrho}) \to H^0(\SL,\pi^*\Pi|_{\Mrho}) \to H^0(\SL,\pi^*\Upsilon|_{\Mrho}) \to 0
\]
is exact.

Hence, one can consider the following part of the associated long exact sequence:
\begin{align}\label{KodairaSpencerGroupCohomology}
    0 \longrightarrow \Hrho^\Gamma  \longrightarrow H^0(\Gamma,H^0(\SL,\pi^*\Pi|_{\Mrho})) \longrightarrow Z^1(\Gamma,\Slrho) \overset{\widetilde{\mathcal{KS_\rho}}}{\longrightarrow} H^1(\Gamma,\Hrho)
\end{align}
where $Z^1(\Gamma,\Slrho)\simeq T_\rho\Rep\simeq H^0(\Gamma,H^0(\SL,\pi^*\Upsilon|_{\Mrho})) $. This exact sequence is isomorphic to the long exact sequence associated to \eqref{exactsequence} in \v Cech cohomology by \cite{Mumford}. All diagrams formed by this isomorphism is commutative and this is why we called $\widetilde{\mathcal{KS_\rho}}$ the map above the Kodaira-Spencer map:
\begin{center}
    \begin{tikzcd}
        \cdots \ar[r] & Z^1(\Gamma,\Slrho) \ar[d,"\rotatebox{90}{\(\sim\)}"] \ar[r,"\widetilde{\mathcal{KS_\rho}}"] & H^1(\Gamma,\Slrho) \ar[d,"\rotatebox{90}{\(\sim\)}"] \ar[r] & \cdots\\
        \cdots \ar[r] & H^0(\Mrho,\Upsilon_{\Mrho}) \ar[r, "\mathcal{KS}"] & H^1(\Mrho,\Theta_\rho) \ar[r] & \cdots
    \end{tikzcd} 
\end{center}

\begin{proposition}\label{projetablevf}
    Let $\rho$ be an admissible representation. Then $$H^0(\Mrho,\Pi|_{\Mrho})\simeq H^0(\SL,\pi^*\Pi|_{\Mrho})^\Gamma \simeq \Sl$$ 
\end{proposition}

\begin{proof}
    We show that the vector space of projectable vector fields, i.e. vector fields that descends to the quotient, is isomorphic to $\Sl$. Let 
    \begin{align*}
        G_\gamma: \widetilde{\mathfrak{X}}& \to  \widetilde{\mathfrak{X}} \\
        (x,\rho) & \mapsto \  (\rho(\gamma^{-1})x\gamma,\rho)
    \end{align*}
    Then a vector field $V$ on $\widetilde{\mathfrak{X}}|_\rho$ is projectable if, and only if, $(G_\gamma)_*V=V,\ \forall \gamma\in\Gamma$. We decompose $V$ in $(v,c)$ where $v$ is a vector field on $\SL$ and $c$ is a cocycle in $Z^1(\Gamma,\Slrho)\simeq T_\rho\Rep^a$.
    From 
    $$(G_\gamma)_*\begin{pmatrix}v(x)\\0 \end{pmatrix}=\begin{pmatrix}Ad_{\rho(\gamma)^{-1}}(v(x))\\0 \end{pmatrix}\text{ and }(G_\gamma)_*\begin{pmatrix}0\\c \end{pmatrix}=\begin{pmatrix} c(\gamma^{-1})\\c \end{pmatrix} $$
    it follows that $V$ is $\Gamma$-invariant, or projectable, if 
    \begin{align}
        c(\gamma^{-1}) +Ad_{\rho(\gamma)^{-1}}(v(x))=v(\rho(\gamma)^{-1}x\gamma)
    \end{align}
    Or equivalently $$c(\gamma)=v(\rho(\gamma)x\gamma^{-1})- Ad_{\rho(\gamma)}(v(x))$$
    Let us remark that if $v$ is constant, then $c$ has to be a coboundary to form a projectable vector field. Hence, for $\gamma, \delta\in \ker(\rho)$, we get
    \begin{align*}
        c((\gamma\delta)^{-1}) & =v(x\gamma\delta)-v(x)\\
        & = c(\gamma^{-1})+c(\delta^{-1}) \\
        & =v(x\gamma)-v(x)+v(x\delta)-v(x)
    \end{align*}
    Putting these lignes together, we obtain
    \begin{align}\label{eqnProjectablevectorfield}
        v(x)= v(x\gamma)+v(x\delta)-v(x\gamma\delta)
    \end{align}
    If $v$ is fixed by the action described by this formula, see as an action of the subgroup $\ker(\rho) \times \ker(\rho)$ in $\SL\times \SL$, then it is fixed by the Zariski closure of it. By lemma \cite[Lemme 5.6]{Ghys}, the Zariski closure of $\ker(\rho)$ is $\SL$ and the Zariski closure of a product is the product of the Zariski closure (true for general topology), so that 
    \[
        \overline{\ker(\rho) \times \ker(\rho)}^{Zar}=\SL\times \SL
    \]
    As constant vector fields together with the corresponding coboundary form a suitable projetable vector field, one can suppose that $v(\Id)=0_{\Sl}$. Thus, \eqref{eqnProjectablevectorfield} implies (for $x=\Id$) that $v$ is an holomorphic morphism from $\SL$ to its Lie algebra.\\
    Finally, as $\Sl$ is an abelian group, $v$ factorizes through 
    \[
        \SL^{ab}=\SL/\left[\SL,\SL \right]=\{\Id\}
    \]
    and $v$ is globally constant and by assumption equal to $0$.
\end{proof}

\begin{proposition}\label{KSsurjective}
    For any $\rho \in \Rep^a$, the Kodaira-Spencer map $\mathcal{KS}_\rho$ is surjective.
\end{proposition}

\begin{proof}
    By \eqref{KodairaSpencerGroupCohomology}, it is equivalent to prove that 
    \[
        \widetilde{\mathcal{KS}_{\rho_0}}:H^0(\pi^*\Upsilon\vert_{\Mrho})^\Gamma \to H^1(\Gamma,\Hrho)
    \]
    is surjective.\\
    By proposition \ref{projetablevf}, we know that $H^0(\pi^*\Pi|_{\Mrho})^\Gamma$ is isomorphic to the image of $\Sl$ under the following map 
    \begin{align*}
        f:\Sl & \to \Sl \oplus B^1(\Gamma,\Slrho)\\
        X & \to (X,\phi:\gamma\mapsto X-Ad_{\rho(\gamma)}X).
    \end{align*}
    We obtain the following diagram of exact sequences 
    \begin{center}
        \begin{tikzcd}[column sep=1.6em]
            0 \ar[r] & \Sl^{\rho(\Gamma)}  \ar[r] \ar[d,"\wr"] & \Sl\oplus B^1(\Gamma,\Slrho) \ar[d,"\wr"] \ar[r] & Z^1(\Gamma,\Slrho) \ar[d,"\wr"] \ar[r, "\widetilde{\mathcal{KS_\rho}}"] & H^1(\Gamma,\Slrho) \ar[d,"\wr"] \\
            0 \ar[r] & \Hrho^\Gamma \ar[r]& H^0(\pi^*\Pi\vert_{\Mrho})^\Gamma \ar[r] & H^0(\pi^*\Upsilon\vert_{\Mrho})^\Gamma \ar[r, "\mathcal{KS_\rho}"] & H^1(\Gamma,\Hrho)
        \end{tikzcd}
    \end{center}
    We conclude that the Kodaira-Spencer map is nothing else than the natural projection $Z^1(\Gamma,\Slrho)\to H^1(\Gamma,\Slrho)$ which is surjective.
\end{proof}

\subsection{Equivariant slices and Kuranishi spaces.}

The geometry of a representation variety is in general very complicated and it is a hard game to find an $\SL$-equivariant slice which, as a germ at $\rho\in\Rep$, gives by proposition \ref{completeness} the Kuranishi space of $\rho$. However, few words can be said in general. 

\begin{corollary}
    Let $\rho\in \Rep^a$, then any complex analytic space transverse to the orbits passing through $\rho$ gives (as a germ) the Kuranishi space of $\Mrho$.
\end{corollary}

\begin{proof}
    The Kodaira-Spencer is surjective and two conjugated representations gives the same manifold $\Mrho$ up to biholomorphism.
\end{proof}

In opposition to this imprecise result, more acurate results can be state in some particular cases. For instance, assume that the first betti number of $\Gamma$ is one and let 
\[
    \langle \gamma_1,\cdots,\gamma_n \vert\, \mathcal{R} \rangle
\]
be a presentation of $\Gamma$ such that the natural projection $p:\Gamma\to \Gamma^{ab}\simeq \Z$ sends $\gamma_1$ to $1$. 

\begin{proposition}
    For all representation $\rho$ such that $\rho(\gamma_1)$ is semi-simple in $\SL$ then there exists an étale slice $V$ at $\rho$ in $\Rep$ such that $T_\rho V$ is isomorphic to $H^1(\Gamma,\Slrho)$.
\end{proposition}

Since the orbit of such representation is closed \cite[Theorem 30]{Sikora}, this proposition is given by the Luna's slice theorem as in \cite[Proposition 2.8]{BenAbdelghani}.

\section{Example.}\label{example}

We give an example which emphasize the main contribution of this this work compare to \cite{Ghys}. The Weeks manifold $M_W$ is known to has the smallest volume among hyperbolic $3$-manifold with first Betti number $0$. Among all its properties, it is compact, closed, oriented, arithmetic which turn it into a particular case of interest in this note. We also have a presentation of its fundamental group
\[
    \pi_1(M_W)=\langle a,b\vert\, a^2b^2a^2b^{-1}ab^{-1}, a^2b^2a^{-1}ba^{-1}b^2 \rangle
\]
and one discrete and faithful $\SL$-representation of this group is given by 
\[
    \eta(a)=\begin{pmatrix} x & 1 \\ 0 & x^{-1} \end{pmatrix}, \qquad \eta(b)=\begin{pmatrix} x & 0 \\ r & x^{-1} \end{pmatrix}
\]
with $r=2-x-x^{-1}$ and $1+2x^2-x^3+2x^4+x^6=0$ (see \cite{Weeks}). Up to the choice of $x$ we fix $\Gamma_W$ to be the image of $\pi_1(M_W)$ under this representation.

With this presentation, it is clear that this manifold has betti number equal to $0$ and it follows that the trivial representation is an isolated point in representation variety and that $\SL/\Gamma_W$ is rigid (in the sense that the Kuranishi space is a point). But it is not globally rigid. 

Consider the case of a representation of $\Gamma_W$ with abelian image. Direct computations gives that, up to conjugation, it is given by
\[
    \displaystyle \rho_{n,m}(a)=\begin{pmatrix} \omega_m & 0 \\ 0 & \omega_m^{-1} \end{pmatrix}, \qquad \rho_{n,m}(b)=\begin{pmatrix} \omega_n & 0 \\ 0 & \omega_n^{-1}  \end{pmatrix}
\]
for $\omega_k=e^{\frac{2ik\pi}{5}}$ and $m,n=0,\cdots, 4$. Those representations $\rho_{n,m}$ have obviously images in a compact subgroup of $\SL$ and the properness criterion \eqref{propernesscondition} implies that the induced action is admissible.

Let treat the non abelian case. It is a general fact that $\rho(a)$ and $\rho(b)$ can be of the form 
\[
    \rho(a)=\begin{pmatrix} x & 1 \\ 0 & x^{-1} \end{pmatrix}, \qquad \rho(b)=\begin{pmatrix} y & 0 \\ r & y^{-1} \end{pmatrix}
\]
with $x,y\neq 0$. According to \cite[p. 24-25]{Weeks}, the second relation in the presentation of $\pi_1(M_W)$ implies that $y$ is either equal to $x$ or $x^{-1}$. Moreover, if $r\neq 0$ and $x$ is not a root of unity, the representation is faithful and it follows that it is not admissible (as remarked in the proof of proposition \ref{projetablevf}). Following the computations of \cite{Weeks}, none of these cases have a solution (except $x^5=1$, $y=1$ and $r=0$ which leads to an abelian representation) and we conclude that the admissible locus of $\mathcal{R}(\Gamma_W)$ contains only representations with abelian images.

By the open criterion of $\mathcal{R}(\Gamma_W)$ in the Teichmüller stack, we know that $\left\vert \mathcal{T}(\SL/\Gamma_W) \right\vert$ (the underlying topological space) contains at least $25$ isolated points as connected components, corresponding to the representations $\rho_{n,m}$, $n,m=0,\cdots,4$.

\begin{remark}
    There are plenty of examples, but few of them are actually "computable" since the complexity in finding a discrete and faithful representation increase with the complexity of the relations. However, we can mention another example: the manifold $v1539(5,1)$ (in the notations of SnapPy). This manifold has a first Betti number of $2$. It is a closed, oriented, compact and hyperbolic $3$-fold which fiber over the circle. We have a presentation of its fundamental group with SnapPy and by same arguments as above, one can show that the quotient stack
    \[
        \left\{A,B\in \SU\times \SU\vert\, [A,B]=\Id \right\}/\SL 
    \]
    is in the Teichmüller space and corresponds to a subset of the set of abelian representations. But it is a hard game to find other admissible representations since the numerical criterion \eqref{propernesscondition} not allow us to do real computations.
\end{remark}

\section*{Acknowledgement}
    I would like to express my deep gratitude to Laurent \textsc{Meersseman} and Marcel \textsc{Nicolau}, my research supervisors, for their patient guidance, enthusiastic encouragement and useful critics of this research work. I would also like to thank Alberto \textsc{Verjovsky}, Joan \textsc{Porti} and Nicolas \textsc{Tholozan} for interesting discussions.

\end{document}